\documentclass{article}

\usepackage{graphicx}
\usepackage{amsmath,amssymb,latexsym}
\usepackage{amsthm}
\usepackage{fullpage}

\newcommand{\norm}[1]{\left\Vert#1\right\Vert}

\newcommand{\dprod}[2]{\left\langle #1,#2 \right\rangle}
\newcommand{\vol}[1]{\text{Vol}\left( #1 \right)}
\newcommand{\bR}{\displaystyle \mathbb{R}}

\newcommand{\bP}{\displaystyle \mathbb{P}}

\newcommand{\cQ}{\displaystyle \mathcal{Q}}
\newcommand{\proj}{\displaystyle \textrm{Proj}}

\newtheorem{theorem}{Theorem}[section]
\newtheorem{Conjecture}[theorem]{Conjecture}
\newtheorem{corollary}[theorem]{Corollary}
\newtheorem{proposition}[theorem]{Proposition}
\newtheorem{lemma}[theorem]{Lemma}

\begin{document}
\title{\Large Sampling on the Sphere by Mutually Orthogonal Subspaces}
\author{Uri Grupel
\thanks{Weizmann Institute of Science. urigrupel@gmail.com}}
\date{}
\maketitle
\begin{abstract}
	\small\baselineskip=9pt
	The purpose of this paper is twofold.
	First, we provide an optimal $\Omega(\sqrt{n})$ bits lower bound for any two-way  protocol for the Vector in Subspace Communication Problem which is of bounded total rank. 
	This result complements Raz's $O(\sqrt{n})$ protocol, which has a simple variant of bounded total rank.
	Second, we present a plausible mathematical conjecture on a measure concentration phenomenon that implies an $\Omega(\sqrt{n})$ lower bound for a general protocol.
	We prove the conjecture for the subclass of sets that depend only on $O(\sqrt{n})$ directions.
\end{abstract}

	\section{Introduction}\label{sec:intro}
	
	The Vector in Subspace Problem ({\it VSP}) is a communication problem where one party (Alice) receives a unit vector $u\in S^{n-1}$, and a second party (Bob) receives a subspace $H\subseteq \bR^n$ of dimension $\lfloor n/2 \rfloor$ such that either $u\in H$ or $u\in H^{\perp}$. The goal of Alice and Bob is to determine whether $u\in H$ or not.\\
	
	VSP was introduced by Kremer \cite{kremer95} and has been studied under both classical and quantum communication models .
	In the classical communication model, Alice and Bob exchange bits between them in order to determine whether $u\in H$. In the quantum communication model, Alice and Bob exchange qubits.\\
	
	In this paper, the terms protocol and complexity refer to distributional complexity (see \cite{yao79} \cite{babai+frankl+simon86}). That is, a protocol outputs the correct answer with probability at least $2/3$. 
	The complexity is measured according to the number of bits or qubits that are exchanged in the worst case .\\
	
	It is known that VSP can be solved in the quantum model with the exchange of $O(\log n)$ qubits. In \cite{raz99} Raz presented a classical protocol that solves VSP with the exchange of $O(\sqrt{n})$ bits. In \cite{klartag+regev11}, Klartag and Regev proved that any classical protocol for VSP has a communication complexity of at least $\Omega(n^{1/3})$ bits. Thus, VSP shows
	that quantum communication can be exponentially stronger than classical communication. In this paper we discuss the gap between the lower and upper bound for the classical model.\\
	
	We focus on the class of protocols of {\it bounded total rank}.
	In a deterministic protocol, each decision by Alice on the value of the next bit to be sent to Bob is based on two factors:
	her knowledge of the communication received so far, and perhaps an additional measurement of the vector $u$.
	We define the rank of the decision to be the number of linear functionals of the vector $u$ that Alice has to compute in order to carry out the measurement.
	For example, deciding the value of the next bit by using the indicator function of the set $\{\dprod{x}{v_1}\geq 0, \sin(\dprod{x}{v_2}^2)\geq 1/2\}$,
	where $v_1,v_2\in\bR^n$ are determined by the communication received so far, is a decision of rank 2. 
	In general, the rank of a decision is an integer between 0 and n.
	The total rank of a protocol is the sum of all ranks of decisions
	made by Alice in the worst case scenario. Note that we do not count decisions by Bob. 
	The protocol Raz introduced can be slightly modified to be of total rank $O(\sqrt{n})$ (for more details see Appendix \ref{app:protocol}).\\
	
	We prove that any protocol of VSP, for the classical model, of total rank at most $O(\sqrt{n})$ has communication complexity of at least $\Omega(\sqrt{n})$ bits. In light of the upper bound by Raz, this lower bound is sharp. We also introduce a novel mathematical conjecture about concentration of measure in the high dimensional sphere. This conjecture implies that any classical protocol for VSP has a communication complexity of at least $\Omega(\sqrt{n})$ bits.\\
	
	The lower bound by Klartag and Regev is a result of a concentration theorem for sampling on the sphere by random subspaces. They proved that for any measurable subset $A\subseteq S^{n-1}$ with $\sigma_{n-1}(A)\geq Ce^{-cn^{1/3}}$, where $\sigma_{n-1}$ is the uniform probability measure on the sphere $S^{n-1}$, it holds that
	$$\bP_{H}\left(\left|\frac{\sigma_H(A\cap H)}{\sigma_{n-1}(A)}-1\right|\leq 0.1\right)\geq 1-e^{-c'n^{1/3}},$$
	where $C,c,c'>0$ are universal constants.
	Here $\sigma_H$ denotes the Haar probability measure on $S^{n-1}\cap H$ and $\bP_H$ denotes the orthogonally invariant Haar probability measure over the Grassmanian manifold of subspaces $H\subseteq \bR^n$ of dimension $\lfloor n/2\rfloor$.\\
	
	The concentration inequality by Klartag and Regev is sharp. Taking $A=\{x\in S^{n-1};\;x_1\geq T\}$, where $T\approx n^{-1/3}$ is chosen such that $\sigma_{n-1}(A)=n^{-1/3}$ gives
	$$\bP_{H}\left(\left|\frac{\sigma_H(A\cap H)}{\sigma_{n-1}(A)}-1\right|\leq 0.1\right)= 1-e^{-c_n n^{1/3}},$$
	where $c_n$ has a finite limit $c\in(0,\infty)$ as $n\rightarrow\infty$.\\
	
	Our goal is to find a concentration inequality that applies to smaller sets, that is sets with measure of the order of magnitude of $e^{-\sqrt{n}}$. Our hope is that by considering both $H$ and $H^{\perp}$ simultaneously, a stronger concentration result can be achieved.\\

	\begin{Conjecture} \label{conj:main}
		Let $A\subseteq S^{n-1}$ be a measurable subset with $\sigma_{n-1}(A)\geq e^{-c\sqrt{n}}$. Then
		$$\bP_{H}\left(\sqrt{\sigma_H(A\cap H)\sigma_{H^{\perp}}(A\cap H^{\perp})}\geq 0.9\sigma_{n-1}(A)\right) \geq 1-Ce^{-c'\sqrt{n}},$$
		where  $C,c,c'>0$ are universal constants.
	\end{Conjecture}

	Conjecture \ref{conj:main} was essentially suggested by Klartag and Regev \cite{klartag11}, albeit with a weaker arithmetic average in place of the geometric one.\\
	
	In \S \ref{sec:evidence} we prove a special case of the conjecture where the set $A\subseteq S^{n-1}$ is of the form $\{x\in S^{n-1};\; (x_1,\ldots,x_k)\in I\}$ for some Borel set $I\subseteq B_k=\{x\in\bR^k;\;|x|\leq 1\}$, and $k=O(\sqrt{n})$. By considering the case $k=1$, this result shows that the conjecture holds for the extremal case of the theorem by Klartag and Regev. This extremal case also shows that if the conjecture is true it is tight.\\
	
	This special case of the conjecture follows from the following result:

	\begin{theorem}\label{thm:kcoordinates}
		Let $k\leq \alpha_1\sqrt{n}$. Let $f:S^{n-1}\to[0,\infty)$ be a measurable function such that $\norm{f}_{\infty}\leq e^{\alpha_2\sqrt{n}}$, $\norm{f}_1=1$ and $f$ depends only on $x_1,\ldots,x_k$. Then,
		$$\bP_H\left(\sqrt{\int_{S_H}f(x)d\sigma_H(x)\int_{S_{H^{\perp}}}f(x)d\sigma_{H^{\perp}}(x)}\geq 0.9\right) \geq 1-\alpha_3e^{-\sqrt{n}},$$
		where $\alpha_1,\alpha_2,\alpha_3>0$ are universal constants.
	\end{theorem}

	In the proof we use various tools from Geometric Functional Analysis. We begin by reformulating the problem in terms of random matrices instead of random subspaces. We show that the event in Theorem \ref{thm:kcoordinates} strongly depends on the singular values of a random projection. Using results from the theory of Wishart matrices we show that these singular values are concentrated around their expected values.\\
	
	Next, we use the Cauchy-Schwarz inequality to define a smaller event than the one in the theorem.
	The use of the Cauchy-Schwarz inequality demonstrates how considering both $H$ and $H^{\perp}$ simultaneously can enhance the concentration results and is fundamental to our approach.\\
	
	Finally, we use the concentration results, and asymptotic tools such as the Laplace method in order to show that with high probability this smaller event holds true. In this last step we present bounds for the universal constants in Theorem \ref{thm:kcoordinates} which are, in principle, explicit.\\
	
	In \S \ref{sec:applications} we employ the rectangle method and show how Theorem \ref{thm:kcoordinates} implies a sharp lower bound for classical protocols of total rank at most $O(\sqrt{n})$.
	
	\begin{corollary}\label{cor:vsp}
		Let $\mathcal{P}$ be a protocol for the {\it Vector in Subspace Problem} of total rank at most $\alpha_1\sqrt{n}$, with probability of error which is at most a constant smaller than $\frac{1}{2}$. Then the communication complexity of $\mathcal{P}$ is $\Omega(\sqrt{n})$ bits.
	\end{corollary}
	
	Using the same methods, a positive resolution of Conjecture \ref{conj:main} would imply a sharp lower bound for a general classical protocol to VSP.
	
	\begin{theorem}\label{thm:vsp}
		Let $\mathcal{P}$ be a general protocol for the {\it Vector in Subspace Problem} with probability of error which is at most a constant smaller than $\frac{1}{2}$. If Conjecture \ref{conj:main} is true then the communication complexity of $\mathcal{P}$ is $\Omega(\sqrt{n})$ bits.
	\end{theorem}
	
	{\it Acknowledgement.} This paper was written under the supervision of Bo'az Klartag whose guidance, support and patience were invaluable.
	In addition, I would like to thank Sasha Sodin for useful discussions and suggestions, and Oded Regev for his remarks on an early draft of this paper. Supported by the European Research Council (ERC).
	
	\section{Applications to VSP}\label{sec:applications}
	
	In this section we prove that Theorem \ref{thm:kcoordinates} implies a lower bound of $O(\sqrt{n})$ for classical protocols of total rank at most $O(\sqrt{n})$. We also show that Conjecture \ref{conj:main} implies a lower bound of $O(\sqrt{n})$ for any classical protocol.
	In light of the result of Raz, if the conjecture is true then this bound is sharp.\\
	
	Theorem \ref{thm:kcoordinates} implies a special case of the conjecture for sets that depend only on $\alpha_1\sqrt{n}$ directions.
	For any such $A\subseteq S^{n-1}$ with $\sigma_{n-1}(A)\geq e^{-\alpha_2\sqrt{n}}$, define $f(x)=1_{A}(x)/\sigma_{n-1}(A)$.
	The function $f$ depends only on $\alpha_1\sqrt{n}$ directions, bounded by $1/\sigma_{n-1}(A)\leq e^{\alpha_2\sqrt{n}}$ and has $\norm{f}_1=1$. Hence, we may apply Theorem \ref{thm:kcoordinates}. We have,

	\begin{align*}
	1-\alpha_3e^{-\sqrt{n}}
	&\leq
	\bP_H\left(\sqrt{\int_{S_H}1_{A}(x)/\sigma_{n-1}(A)d\sigma_H(x)}
	\sqrt{\int_{S_{H^{\perp}}}1_{A}(x)/\sigma_{n-1}(A)d\sigma_{H^{\perp}}(x)}\geq 0.9\right)\\
	&=\bP_H\left(\sqrt{\sigma_H(A\cap H)\sigma_{H^{\perp}}(A\cap H^{\perp})}/\sigma_{n-1}(A)\geq 0.9\right).
	\end{align*}

	In this section, we use this consequence of Theorem \ref{thm:kcoordinates}.\\

	Our argument follows the rectangle method usually attributed to Babai, Frankl and Simon \cite{babai+frankl+simon86} and Razborov \cite{razborov11}.\\
	
	For simplicity, we assume that $n$ is even. We denote by $G_{n/2}$ the Grassmanian manifold of all subspaces of $\bR^n$ of dimension $n/2$, equipped with the $O_n$ invariant measure $\sigma_G$. Let $\mu_0$ be the uniform measure on $S^{n-1}\times G_{n/2}$.
	Denote 
	$$I_1=\left\{(u,H)\in S^{n-1}\times G_{n/2};\; u\in H\right\},$$ 
	and 
	$$I_2=\left\{(u,H)\in S^{n-1}\times G_{n/2};\; u\in H^{\perp}\right\}.$$ 
	Let $\mu_i$ be the Haar invariant probability measure on $I_i$ for $i=1,2$, with respect to the obvious $O_n$ action. Such measure exists due to the transitive property of such action. For a rectangular set $A_i\times B_i\subseteq I_i$ we have
	$$\mu_1(A_1\times B_1)=\int_{H\in B_1} \sigma_H(A_1\cap H),$$
	and
	$$\mu_2(A_2\times B_2)=\int_{H\in B_2} \sigma_{H^{\perp}}(A_2\cap H^{\perp}).$$
	By replacing $\alpha_2$ in Theorem \ref{thm:kcoordinates} with $\min\{\alpha_2,1\}$ we may assume that it is at most $1$.

	\begin{proposition}\label{prop:big_sets}
		Let $Q=A\times B\subseteq S^{n-1}\times G_{n/2}$ be such that $\mu_0(A\times B)\geq Ce^{-\alpha_2\sqrt{n}}$. Assume that the set $A$ depends on $\alpha_1\sqrt{n}$ directions. Then
		$$\sqrt{\mu_1(Q)\mu_2(Q)}\geq 0.8\mu_0(Q).$$
		The constants $\alpha_1,\alpha_2,C>0$ are universal constants.
	\end{proposition}

	\begin{proof}
		Define 
		\begin{align*}
			E=\left\{ H\in B;\;\sqrt{\sigma_H(A\cap H)\sigma_{H^{\perp}}(A\cap {H^{\perp}})}\leq 0.9\sigma(A)\right\}.
		\end{align*}
	
		According to our assumption $\sigma_{n-1}(A)\geq \mu_0(A\times B)\geq Ce^{-\alpha_2\sqrt{n}}$.
		By the Theorem \ref{thm:kcoordinates}, $\bP(E)\leq \alpha_3e^{-\sqrt{n}}$. 
		We choose $C\geq 1$ big enough, such that
		\begin{equation}\label{eq:constant}
		0.9\sigma_G(B\setminus E)\geq 0.8\sigma_G(B).
		\end{equation}
		
		By the Cauchy-Schwarz inequality
		\small
		\begin{align*}
		\sqrt{\mu_1(Q)\mu_2(Q)}
		&=\sqrt{\left(\int_{H\in B}\sigma_{H\in B}(A\cap H) \right)\left(\int_{H\in B}\sigma_{H^{\perp}}(A\cap {H^{\perp}}) \right)}
		\geq \int_{H\in B}\sqrt{\sigma_H(A\cap H)\sigma_{H^{\perp}}(A\cap {H^{\perp}})}\\
		&\geq \int_{H\in B\setminus E}\sqrt{\sigma_H(A\cap H)\sigma_{H^{\perp}}(A\cap {H^{\perp}})}
		\geq 0.9 \int_{H\in B\setminus E}\sigma(A)=0.9\sigma(A)\sigma_G(B\setminus E).
		\end{align*}
		\normalsize
		Equation (\ref{eq:constant}) gives us
		$$\sqrt{\mu_1(Q)\mu_2(Q)}\geq 0.8\sigma_G(B)\sigma(A)=0.8\mu_0(Q).$$
	\end{proof}
	
	\begin{corollary} \label{cor:rectangles}
		Let $Q=A\times B\subseteq S^{n-1}\times G_{n/2}$. Assume that the set $A$ depends on $\alpha_1\sqrt{n}$ directions. Then
		$$\sqrt{\mu_1(Q)\mu_2(Q)}\geq 0.8\mu_0(Q)-Ce^{-\alpha_2\sqrt{n}}.$$
	\end{corollary}
	
	Using the above propositions, we are ready to prove Corollary \ref{cor:vsp}.
	
	\begin{proof}
		By repeated application of the protocol we may assume that the probability of error is less then $\frac{1}{9}$.
		By Yao's principle \cite{yao77}, we may assume that our protocol is a randomly chosen deterministic protocol.
		Let $D$ be the number of bits exchange in the protocol.
		We have a partition of $S^{n-1}\times G_{n/2}$ into $2^D$ rectangles of the form $Q=A\times B$, each labeled as ``In $H$'' or ``In $H^{\perp}$''. 
		Since we assume the total rank is at most $\alpha_1\sqrt{n}$, for every $Q=A\times B$ in the partition, the set $A$ is determined by at most $\alpha_1\sqrt{n}$ directions.
		Denote by $\cQ_+$ all the rectangles labeled ``In $H$'', and by $\cQ_-$ all the rectangles labeled ``In $H^{\perp}$''.
		According to our assumption $\sum_{Q\in \cQ_+} \mu_2(Q)\leq \frac{1}{9}$ and $\sum_{Q\in \cQ_-} \mu_1(Q)\leq \frac{1}{9}$.
		By Corollary \ref{cor:rectangles} and the Cauchy-Schwarz inequality, we have
		\begin{align*}
			\sum_{Q\in \cQ_+} (0.8\mu_0(Q)-Ce^{-\alpha_2\sqrt{n}})
			&\leq \sum_{Q\in \cQ_+}\sqrt{\mu_1(Q)\mu_2(Q)}
			\leq \sqrt{\left(\sum_{Q\in \cQ_+}\mu_1(Q)\right)\left(\sum_{Q\in \cQ_+}\mu_2(Q)\right)}\\
			&\leq \sqrt{1\cdot\frac{1}{9}}
			=\frac{1}{3}.
		\end{align*}
		Similarly we have,
		$$\sum_{Q\in \cQ_-} (0.8\mu_0(Q)-Ce^{-\alpha_2\sqrt{n}})\leq \frac{1}{3}.$$
		Summing the above inequalities, we obtain
		$$0.8-2^DCe^{-\alpha_2\sqrt{n}}\leq \frac{2}{3}\Rightarrow D\geq C'\sqrt{n}.$$
	\end{proof}
	
	If Conjecture \ref{conj:main} is true, then Proposition \ref{prop:big_sets} and Corollary \ref{cor:rectangles} are true without the assumption that the set $A$ depends on $\alpha_1\sqrt{n}$ directions. Hence, we may repeat the proof of Corollary \ref{cor:vsp} for a general classical protocol, and deduce Theorem \ref{thm:vsp}.

	\section{Proofs}\label{sec:evidence}
	In this section we prove Theorem \ref{thm:kcoordinates}. This theorem is a special case of the conjecture for functions that depend only on $O(\sqrt{n})$ directions. We assume that $n$ is even and greater than some universal constant.
	In this section, when we say {\it uniform distribution}, we refer to the Haar probability distribution.\\ 
	
	Throughout this section we shall use the letters $c, \tilde{c}, C$ etc. to denote various universal constants,
	whose value may change from one line to the next. Additionally, $\alpha_1,\alpha_2, \alpha_3, \rho > 0$
	are universal constants whose value would be determined only at the end of the section. Specifically, in this section
	we will assume a few upper bounds for $\alpha_1$ in terms of explicit positive universal constants,
	an upper bound for $\alpha_2$ in terms of $\alpha_1$, and a lower bound for $\alpha_3$ in terms of $\alpha_2$. \\
	
	Let $\psi:B_k\times S^{m-k-1}\to S^{m-1}$ be defined by $\psi(x,y)=(x,\sqrt{1-|x|^2}y)$. The map $\psi$ enables us to separate the dependence on the first $k$ coordinates. The following change of variables formula is standard:
	\begin{proposition}\label{prop:coarea}
		For any integrable $f:S^{m-1}\to\bR$ and any $1\leq k \leq m-1$ there exists $C_{m,k}=(m-k)\textrm{Vol}(B_{m-k})/(m\textrm{Vol}(B_m))$ such that
		\begin{align*}
			\int_{S^{m-1}}f(x)d\sigma_{m-1}(x)
			=C_{m,k} \int_{B_k}\left(1-|x|^2\right)^{(m-k-2)/2}
			\int_{S^{m-k-1}}f\left(x,\sqrt{1-|x|^2}\theta\right)d\sigma_{m-k-1}(\theta)dx.
		\end{align*}
	\end{proposition}
	
	Let $E=\textrm{span}\{e_1,\ldots,e_k\}$. Let $H\subseteq\bR^n$ be a random subspace of dimension $n/2$ distributed uniformly.
	Let $\lambda_1,\ldots,\lambda_k$ be the singular values of the projection map $P:H\to E$. The singular values $\lambda_1,\ldots,\lambda_k$ are the cosines of the principal angles between $H$ and $E$. 
	The next proposition along with Proposition \ref{prop:coarea}, allows us to study the distribution of singular values of a random matrix instead of integration on a random subspace.
	
	\begin{proposition} \label{prop:distribution}
		Let $H$ be a random subspace of dimension $n/2$, let $E$ and $\lambda_1,\ldots,\lambda_k$ be as before. Let $\Lambda=\textrm{diag}(\lambda_1,\ldots,\lambda_k)$. Let $U:E\to E$ be a random orthogonal map distributed uniformly, independent of $H$. Let $f:S^{n-1}\to\bR$ be a measurable function such that $f$ depends only on the first $k$ coordinates. Then, the random variable $$\sqrt{\int_{S_H}f(x)d\sigma_H(x)\int_{S_{H^{\perp}}}f(x)d\sigma_{H^{\perp}}(x)}$$
		is equal in distribution to
		\begin{align} \label{formula:fixed_space}
		C_{n/2,k}\sqrt{\int_{B_k}f\left(U\Lambda U^T x\right)\left(1-|x|^2\right)^{(n/2-k-2)/2}dx}
		\sqrt{\int_{B_k}f\left(U\sqrt{I-\Lambda^2} U^Tx\right)\left(1-|x|^2\right)^{(n/2-k-2)/2}dx}.
		\end{align}
		The constant $C_{n/2,k}$ is the same as in Proposition \ref{prop:coarea} with $m=n/2$.
	\end{proposition}
	
	\begin{proof}
		In this proof, we construct an orthogonal map $V:\bR^n\to\bR^n$ that maps $H$ and $H^{\perp}$ to canonical subspaces that depend only on the principle angles and a rotation $U:E\to E$. The map $V$ is chosen such that $f$ would be invariant under $V$. In order to construct $V$ we use the projection maps to define appropriate orthonormal bases for $H$, $H^{\perp}$ and $E$.\\
		By the Singular Value Decomposition ({\it SVD}) of the projection $P:H\to E$ there exists an orthonormal basis $x_1,\ldots,x_k$ of $E$ and an orthonormal basis $y_1,\ldots,y_{n/2}$ of H such that 
		$$P=\sum_{i=1}^k\lambda_ix_i\otimes y_i.$$
		Since $Py_j=0$ for all $j=k+1,\ldots,n/2$ we have $y_{k+1},\ldots,y_{n/2}\in E^{\perp}$. Since $Py_i = \lambda_i x_i$ for $i=1,\ldots,k$ there exists a unit vector $v_{n/2+i}\in E^{\perp}$ such that
		$$y_i=\lambda_ix_i+\sqrt{1-\lambda_i^2}v_{n/2+i}, \quad \forall i=1,\ldots,k.$$
		Denote $v_j=y_j$ for $j=k+1,\ldots,n/2$. 
		Let $i'=n/2+i$ where $i\leq k$ and let $k+1\leq j \leq n/2$. Since $v_{k+1},\ldots,v_{n/2+k}\in E^{\perp}$, we have
		\begin{align*}
			0&=\dprod{y_{i}}{y_j}=\dprod{\lambda_{i}x_{i}+\sqrt{1-\lambda_{i}^2}v_{i'}}{v_j}\\
			&=\sqrt{1-\lambda_{i}^2}\dprod{v_{i'}}{v_j}.
		\end{align*}

		With probability $1$  we have $0<\lambda_i<1$, hence with probability $1$
		$$\dprod{v_{i'}}{v_j}=0.$$
		By the same argument we obtain $\dprod{v_{i}}{v_j}=0$ for all $i\ne j$.
		Let $P_{\perp}:H^{\perp}\to E$ be the orthogonal projection to $E$. The singular values of $P_{\perp}$ are exactly 
		$\sqrt{1-\lambda_1^2},\ldots,\sqrt{1-\lambda_k^2}$.
		Note that $$\sqrt{1-\lambda_1^2}x_1-\lambda_1v_{n/2+1},\ldots,\sqrt{1-\lambda_k^2}x_k-\lambda_kv_{n/2+k}\in H^{\perp}$$ 
		are orthogonal to each other.
		Since the SVD is unique, up to trivial transformations, we have $$P_{\perp}=\sum_{i=1}^k\sqrt{1-\lambda_i^2}x_i\otimes \left(\sqrt{1-\lambda_i^2}x_i-\lambda_iv_{n/2+i}\right).$$
		Hence, there exist $v_{n/2+k+1},\ldots,v_n\in E^{\perp}$ such that
		\begin{align*}
			&\sqrt{1-\lambda_1^2}x_1-\lambda_1v_{n/2+1},\ldots,\sqrt{1-\lambda_k^2}x_k-\lambda_kv_{n/2+k},
			v_{n/2+k+1},\ldots,v_n
		\end{align*}
		is an orthonormal basis of $H^{\perp}$.
		By the same argument as before, with probability one, $v_{n/2+1},\ldots,v_n$ are orthogonal to each other.
		Since $v_{n/2+k+1},\ldots,v_n\in H^{\perp}$ and $v_{k+1},\ldots,v_{n/2}\in H$ we find that
		$x_1,\ldots,x_k,$ $v_{k+1},\ldots,v_n$ is an orthonormal basis of $\bR^n$.
		Let $V$ be the orthogonal map defined by $Vx_i=x_i$ for $i=1,\ldots,k$ and $Vv_j=e_j$ for $j=k+1,\ldots,n$. Since $V$ is the identity map on $E$ we have $f(Vx)=f(x)$ for all $x\in S^{n-1}$. Hence,
		\begin{align*} 
		\sqrt{\int_{S_H}f(x)d\sigma_H(x)\int_{S_{H^{\perp}}}f(x)d\sigma_{H^{\perp}}(x)}
		&=\sqrt{\int_{S_H}f(Vx)d\sigma_H(x)\int_{S_{H^{\perp}}}f(Vx)d\sigma_{H^{\perp}}(x)}\\
		&=\sqrt{\int_{S_{VH}}f(x)d\sigma_{VH}(x)\int_{S_{VH^{\perp}}}f(x)d\sigma_{VH^{\perp}}(x)}
		\end{align*}
		Let $B_{H,k}$ be the unit ball of 
		\begin{align*}
			V\left(H\cap(E^{\perp}\cap H)^{\perp}\right)=
			\textrm{span}\{\lambda_1x_1+\sqrt{1-\lambda_1^2}e_{n/2+1},\ldots,\lambda_kx_k+\sqrt{1-\lambda_k^2}e_{n/2+k}\}.
		\end{align*}

		For any 
		$$x=\sum_{i=1}^{k}t_i\left(\lambda_ix_i+\sqrt{1-\lambda_i^2}e_{n/2+i}\right)\in B_{H,k},$$ 
		where $\sum_{i=1}^kt_i^2\leq 1$, we have
		\begin{align*}
			f(x)&=f\left(\sum_{i=1}^{k}t_i\left(\lambda_ix_i+\sqrt{1-\lambda_i^2}e_{n/2+i}\right)\right)
			=f\left(\sum_{i=1}^{k}t_i\lambda_ix_i\right).
		\end{align*}
		Let $U:E\to E$ be the orthogonal map defined by $Ux_i=e_i$ for $i=1,\ldots,k$. We have,
		$$\int_{B_{H,k}}f(x)dx=\int_{B_k}f(U\Lambda U^Tx)dx,$$
		where $B_k$ is the unit ball of $E$. 
		Since the distribution of $H$ is invariant under the action of $O(k)\times O(n-k)$, the distribution of $U$ is uniform over the orthogonal maps of $E$.
		Let $B_{H^{\perp},k}$ be the unit ball of $\textrm{span}\{\sqrt{1-\lambda_1^2}x_1-\lambda_1e_{n/2+1},\ldots,\sqrt{1-\lambda_k^2}x_k-\lambda_ke_{n/2+k}\}$.
		Using the same map $U$, we have
		$$\int_{B_{H^{\perp},k}}f(x)dx=\int_{B_k}f(U\sqrt{I-\Lambda^2} U^Tx)dx.$$
		To finish the proof we use Proposition \ref{prop:coarea} on $S_{VH}$ and $S_{VH^{\perp}}$.
	\end{proof}
	
	With probability $1$, the matrices $\Lambda$ and $\sqrt{I-\Lambda^2}$ are invertible. Hence, using the change of variables formula, (\ref{formula:fixed_space}) can be written as
	
	\begin{align*}
		\frac{C_{n/2,k}}{\sqrt{\prod_{j=1}^k\lambda_j\sqrt{1-\lambda_j^2}}}
		&\sqrt{\int_{\bR^k}f\left(x\right)\left(1-\left|\Lambda^{-1}U^T x\right|^2\right)_+^{(n/2-k-2)/2}dx}\\
		&\times\sqrt{\int_{\bR^k}f\left(x\right)\left(1-\left|\left(I-\Lambda^2\right)^{-1/2}U^Tx\right|^2\right)_+^{(n/2-k-2)/2}dx}.
	\end{align*}

	By the Cauchy-Schwarz inequality this is at least

	\begin{align} \label{equation:cs}
	\frac{C_{n/2,k}}{\sqrt{\prod_{j=1}^k\lambda_j\sqrt{1-\lambda_j^2}}}
	\int_{\bR^k}f\left(x\right)\left(1-\left|\Lambda^{-1}U^T x\right|^2\right)_+^{(n/2-k-2)/4}
	\left(1-\left|\left(I-\Lambda^2\right)^{-1/2}U^Tx\right|^2\right)_+^{(n/2-k-2)/4}dx.
	\end{align}

	The random variables $\lambda_1,\ldots,\lambda_k$ are the singular values of a block of size $n/2\times k$ in a random orthogonal matrix. These singular values can be described using Wishart matrices \cite{edleman+sutton07}
	
	\begin{proposition}\label{prop:singular}
		Let $N_1, N_2$ be $(n/2)\times k$ independent random matrices with independent standard Gaussian entries.
		Let $X$ be a random orthogonal matrix, chosen by the Haar uniform distribution. Let
		\small
		$$X=\begin{pmatrix}
		X_{1,1} & X_{1,2}\\
		X_{2,1} & X_{2,2}
		\end{pmatrix}$$
		\normalsize
		Where $X_{1,1}$ is $(n/2)\times k$ block. Then, the singular values of $X_{1,1}$ have the same distribution as the square roots of the eigenvalues of $N_1^TN_1(N_1^TN_1+N_2^TN_2)^{-1}$.
	\end{proposition}
	
	Upper and lower bounds for the eigenvalues of the above matrix, can be achieved using a concentration result by Gordon for singular values of Gaussian matrices \cite{vershynin10}.
	
	\begin{lemma}\label{lemma:singular_concentration}
		Let $A$ be $(n/2)\times k$ random matrix with independent standard Gaussian entries. Assume that $k\leq n/2$. Let $s_1\leq \cdots\leq s_k$ be the singular values of $A$, then with probability greater than $1-2e^{-t^2/2}$ we have
		$$\sqrt{n/2}-\sqrt{k}-t\leq s_1 \leq s_k \leq \sqrt{n/2}+\sqrt{k}+t.$$
	\end{lemma}
	
	Combining both results, we have
	\begin{proposition} \label{prop:singular2}
		Let $k\leq \alpha_1\sqrt{n}$ and let $\lambda_1,\ldots,\lambda_k$ be as before. Then, with probability greater then $1-4e^{-\sqrt{n}}$, we have
		$$\left|\lambda_i-\frac{1}{\sqrt{2}}\right|\leq \frac{C(\sqrt{\alpha_1}+\sqrt{2})}{n^{1/4}}, \; \forall i=1,\ldots,k.$$
	\end{proposition}
	
	\begin{proof}
		Let $N_1, N_2$ be as in Proposition \ref{prop:singular}. By Lemma \ref{lemma:singular_concentration} there exists $\mu_1,\ldots,\mu_k$, $\sigma_1,\ldots,\sigma_k$ and $U,V$ orthogonal matrices such that
		\begin{align*}
			N_1^TN_1 &= U\textrm{diag}(\mu_1^2,\dots,\mu_k^2)U^T,\;N_2^TN_2
			= V\textrm{diag}(\sigma_1^2,\dots,\sigma_k^2)V^T,
		\end{align*}
		and, there exists $C'>0$ such that, with probability greater then $1-4e^{-\sqrt{n}}$, 
		\begin{equation} \label{eq:singular}
		\left|\mu_i^2-\frac{n}{2}\right|,\left|\sigma_i^2-\frac{n}{2}\right|\leq C'(\sqrt{\alpha_1}+\sqrt{2})n^{3/4},
		\end{equation}
		for all $i=1,\ldots,k$.
		Assume that event (\ref{eq:singular}) holds true.
		Let $E_1, E_2$ be defined by
		$N_1^TN_1 = (n/2)I + E_1$ and $N_2^TN_2 = (n/2)I + E_2$. 
		Then $\norm{E_i}_{op}\leq C_i(\sqrt{\alpha_1}+\sqrt{2}) n^{3/4}$ for $i=1,2$. We have,
		\begin{align*}
			&N_1^TN_1(N_1^TN_1+N_2^TN_2)^{-1}
			=\frac{1}{2}\left(I+\frac{2}{n}E_1\right)\left(I+\frac{1}{n}(E_1+E_2)\right)^{-1}.
		\end{align*}
		Let $T_1=(E_1+E_2)/n$ and $T_2=E_1/n$, then $\norm{T_i}_{op}\leq C_i'(\sqrt{\alpha_1}+\sqrt{2})n^{-1/4}$ for $i=1,2$. We have
		\begin{align*}
			&N_1^TN_1(N_1^TN_1+N_2^TN_2)^{-1}
			=\left(\frac{1}{2}I+T_2\right)\left(I-T_1+\sum_{j=2}^{\infty}(-1)^jT_1^j\right).
		\end{align*}
		Hence,
		$$N_1^TN_1(N_1^TN_1+N_2^TN_2)^{-1}=\frac{1}{2}I+T,$$
		where $\norm {T}_{op}\leq \tilde{C}(\sqrt{\alpha_1}+\sqrt{2})n^{-1/4}$.
	\end{proof}
	
	\begin{corollary} \label{cor:sum}
		With probability greater then $1-4e^{-\sqrt{n}}$ we have
		\begin{align*}
			\left|\left|U\Lambda^{-1}U^Tx\right|^2+\left|U(I-\Lambda^{2})^{-1/2}U^Tx\right|^2-4|x|^2\right|
			\leq C(\sqrt{\alpha_1}+\sqrt{2})^2|x|^2/\sqrt{n}, \quad \forall x\in\bR^k.
		\end{align*}
	\end{corollary}
	
	\begin{proof}
		Assume that 
		$$\Lambda = \frac{1}{\sqrt{2}}I+T,$$
		where $\norm{T}_{op}\leq C'(\sqrt{\alpha_1}+\sqrt{2})/n^{1/4}$.
		By the above proposition, this event has probability greater than $1-4e^{-\sqrt{n}}$.
		We have,
		\begin{align*}
			\Lambda^{-2}+(I-\Lambda^2)^{-1}&=\left(\frac{1}{2}I+\sqrt{2}T+T^2\right)^{-1}
			+\left(\frac{1}{2}I-\sqrt{2}T-T^2\right)^{-1}\\
			&=4(I-4(\sqrt{2}T+T^2)^2)^{-1}=4I+\tilde{T},
		\end{align*}
		where $\norm{\tilde{T}}_{op}\leq \tilde{C}\norm{T^2}_{op}\leq C(\sqrt{\alpha_1}+\sqrt{2})^2/\sqrt{n}$. Hence,
		\begin{align*}
			\left|U\Lambda^{-1}U^Tx\right|^2+\left|U(I-\Lambda^{2})^{-1/2}U^Tx\right|^2
			&=\dprod{U(\Lambda^{-2}+(I-\Lambda^2)^{-1})U^Tx}{x}\\
			&=4|x|^2+\dprod{U\tilde{T}U^Tx}{x}.
		\end{align*}
		Hence,
		\begin{align*}
			\left|\left|U\Lambda^{-1}U^Tx\right|^2+\left|U(I-\Lambda^{2})^{-1/2}U^Tx\right|^2-4|x|^2\right|
			\leq \norm{\tilde{T}}_{op}|x|^2\leq C(\sqrt{\alpha_1}+\sqrt{2})^2|x|^2/\sqrt{n}.
		\end{align*}
		
	\end{proof}
	
	The above proof demonstrates how considering both $H$ and $H^{\perp}$ simultaneously can cancel the first order term in concentration inequalities. This cancellation leads to great improvement of the estimations, and it is one of the fundamental ideas of our approach.\\
	
	The concentration of the principal angles, allows us to evaluate the coefficient $C_{n/2,k}\left(\prod_{j=1}^{k}\lambda_j\sqrt{1-\lambda_j^2}\right)^{-1/2}$ and the integral at (\ref{equation:cs}).
	
	\begin{proposition} \label{prop:det_coefficients}
		Let $C_{n,k}$ and $C_{n/2,k}$ be the constants from Proposition \ref{prop:coarea} with $m=n,n/2$. For $k\leq \alpha_1\sqrt{n}$, we have
		$$2^{k/2}\frac{C_{n/2,k}}{C_{n,k}}\geq Ce^{-\alpha_1^2/4}.$$
	\end{proposition}
	
	\begin{proof}
		By the definition of $C_{n,k}$ and $C_{n/2,k}$, we need to estimate
		$$2^{k/2}\frac{\Gamma(n/4+1/2)\Gamma(n/2-k/2+1/2)}{\Gamma(n/4-k/2+1/2)\Gamma(n/2+1/2)}$$
		Using Sterling’s formula and the assumption on $k$, this is
		$$\left(1+O\left(\frac{1}{\sqrt{n}}\right)\right) \textrm{exp}\left(-\frac{k^2}{4n}+O\left(\frac{1}{\sqrt{n}}\right)\right).$$
	\end{proof}
	
	Hence, by choosing $\alpha_1$ small enough and using the concentration result for $\lambda_i\sqrt{1-\lambda_i^2}$ (as in Corollary \ref{cor:sum}) we have:
	
	\begin{corollary} \label{cor:coefficient}
		Let $k\leq \alpha_1\sqrt{n}$ and let $\lambda_1,\ldots,\lambda_k$ be as before. Then, with probability greater than $1-4e^{-\sqrt{n}}$, we have
		$$C_{n/2,k}\sqrt{\prod_{j=1}^{k}\frac{1}{\lambda_j\sqrt{1-\lambda_j^2}}}\geq 0.98 C_{n,k}.$$
	\end{corollary}
	
	\begin{proof}
		Assume that for all $1\leq i\leq k$ we have $\lambda_i^2=1/2+t_i$ where $|t_i|\leq C'(\sqrt{\alpha_1}+\sqrt{2})/n^{1/4}$. By Proposition \ref{prop:singular2}, this event has probability greater than $1-4e^{-\sqrt{n}}$. We have,
		$$\frac{1}{\lambda_i\sqrt{1-\lambda_i^2}}=\sqrt{\frac{1}{(1/2+t_i)(1/2-t_i)}}=\frac{2}{\sqrt{1-4t_i^2}}.$$
		Hence,
		$$\left|\frac{1}{\lambda_i\sqrt{1-\lambda_i^2}}-2\right|\leq \frac{C(\sqrt{\alpha_1}+\sqrt{2})^2}{\sqrt{n}}.$$
		We have
		\begin{align*}
			&\sqrt{\prod_{j=1}^{k}\frac{1}{\lambda_j\sqrt{1-\lambda_j^2}}}\geq 2^{k/2}\textrm{exp}\left(-\frac{C(\sqrt{\alpha_1}+\sqrt{2})^2k}{4\sqrt{n}}+O\left(\frac{k(\sqrt{\alpha_1}+\sqrt{2})^2}{n}\right)\right).
		\end{align*}
		We may assume that $\alpha_1$ is small enough, such that both
		\begin{align*}
			&\textrm{exp}\left(-\frac{C(\sqrt{\alpha_1}+\sqrt{2})^2\alpha_1}{4}+O\left(\frac{\alpha_1(\sqrt{\alpha_1}+\sqrt{2})^2}{\sqrt{n}}\right)\right)
			\geq 0.99,	
		\end{align*}
		and (By Proposition \ref{prop:det_coefficients}),
		$0.99 C_{n/2,k}2^{k/2} \geq 0.98 C_{n,k}$.
		Hence,
		$$C_{n/2,k}\sqrt{\prod_{j=1}^{k}\frac{1}{\lambda_j\sqrt{1-\lambda_j^2}}}\geq 0.99 C_{n/2,k}2^{k/2}\geq 0.98C_{n,k}.$$
	\end{proof}
	
	In order to understand the integral in (\ref{equation:cs}), we write $\bR^k$ as $\rho n^{-1/4}B_k\cup (\bR^k\setminus \rho n^{-1/4}B_k)$ where $\rho>0$. Inside the ball $\rho n^{-1/4}B_k$ the integral is close to $1$. Outside the ball, we show that the integral is negligible.\\

	The estimation inside the ball of radius $\rho n^{-1/4}$ uses standard inequalities and corollary \ref{cor:sum} (see Appendix \ref{app:laplace} for the proof). In this proposition we define an upper bound on $\rho$.
	
	\begin{proposition} \label{prop:small_ball}
		Let $k\leq \alpha_1\sqrt{n}$. Let $\Lambda$ and $U$ be as before. Then with probability greater than $1-4e^{-\sqrt{n}}$
		\begin{align*}
		\int_{\rho n^{-1/4}B_k}&f(x)\left(1-\left|\Lambda^{-1}U^T x\right|^2\right)_+^{(n/2-k-2)/4}
		\left(1-\left|\left(I-\Lambda^2\right)^{-1/2}U^Tx\right|^2\right)_+^{(n/2-k-2)/4}dx\\
		&\geq 0.95 \int_{\rho n^{-1/4}B_k}f(x)\left(1-\left| x\right|^2\right)_+^{(n-k-2)/2}dx.
		\end{align*}
	\end{proposition}

	Using the Laplace method (see Appendix \ref{app:laplace}), we estimate the integral outside $\rho n^{-1/4}B_k$.
	
	\begin{proposition} \label{prop:large_dev}
		Let $k\leq \alpha_1\sqrt{n}$. Then, for any $f:\bR^n\to\bR_+$ with
		$\norm{f}_{\infty}\leq e^{\alpha_2\sqrt{n}}$, we have
		\begin{align*}
			I&=C_{n,k}\int_{\bR^k\setminus \rho n^{-1/4}B_k}f(x)\left(1-|x|^2\right)_+^{(n-k-2)/2}dx
			\leq \frac{2\alpha_1}{\rho^2}e^{-\alpha_2\sqrt{n}}.
		\end{align*}
		The constant $C_{n,k}$ is the same as in Proposition \ref{prop:coarea}, and $\rho>0$ is the same as in Proposition \ref{prop:small_ball}.
	\end{proposition}
	
	\begin{proof}[Proof of Theorem \ref{thm:kcoordinates}]
		By Proposition \ref{prop:distribution} and the Cauchy-Schwarz inequality, the event
		$$\sqrt{\int_{S_H}f(x)d\sigma_H(x)\int_{S_{H^{\perp}}}f(x)d\sigma_{H^{\perp}}(x)}\geq 0.9$$
		has the greater probability than the event
		\begin{align*}
			\frac{C_{n/2,k}}{\sqrt{\prod_{j=1}^k\lambda_j\sqrt{1-\lambda_j^2}}}&
			\int_{\bR^k}f\left(x\right)\left(1-\left|\Lambda^{-1}U^T x\right|^2\right)_+^{(n/2-k-2)/4}
			\left(1-\left|\left(I-\Lambda^2\right)^{-1/2}U^Tx\right|^2\right)_+^{(n/2-k-2)/4}dx
			&\geq 0.9.
		\end{align*}
		
		By Corollary \ref{cor:coefficient} and Proposition \ref{prop:small_ball} with probability greater than $1-4e^{-\sqrt{n}}$ the left hand side is at least
		$$0.93C_{n,k}\int_{\rho n^{-1/4}B_k}f(x)(1-|x|^2)^{(n-k-2)/2}_+dx.$$
		By Proposition \ref{prop:coarea} this is equal to
		\begin{align*}
			0.93\left(\int_{S^{n-1}}f(x)d\sigma_{n-1}(x)-
			C_{n,k}\int_{\bR^k\setminus\rho n^{-1/4}B_k}f(x)(1-|x|^2)^{(n-k-2)/2}_+dx\right).
		\end{align*}
		By Proposition \ref{prop:large_dev} there exists $\hat{C}>0$ such that for all $n>\hat{C}$ we have
		\begin{align*}
			&0.93C_{n,k}\int_{\bR^k\setminus\rho n^{-1/4}B_k}f(x)(1-|x|^2)^{(n-k-2)/4}dx
			\leq 0.01.
		\end{align*}
		Hence,
		\begin{align*}
			&0.93\left(\int_{S^{n-1}}f(x)d\sigma_{n-1}(x)-\right.
			\left.C_{n,k}\int_{\bR^k\setminus\rho n^{-1/4}B_k}f(x)(1-|x|^2)^{(n-k-2)/2}_+dx\right)
			\geq 0.9.
		\end{align*}
	\end{proof}
	
	\appendix
	\section{Protocol for VSP}\label{app:protocol}
	The protocol we present here is a simple modification of the one presented by Raz \cite{raz99}.\\
	As before, $c,c_1,C$ etc. denote positive universal constants.\\
	Let $k=\lfloor e^{\sqrt{n}}\rfloor$. Let $E_1,\dots,E_k\subseteq\bR^n$ be independent random subspaces of dimension $\lfloor C_1\sqrt{n}\rfloor$ chosen uniformly.
	For every $1\leq i \leq k$, let $\mathcal{N}_i=\{\theta^i_1,\dots,\theta^i_m\}$ be independent random vectors in $S^{n-1}\cap E_i$, where $m=\lfloor e^{C_2\sqrt{n}}\rfloor$. Alice and Bob sample $(E_1,\mathcal{N}_1),\dots,(E_k,\mathcal{N}_k)$ in advance and store the results. Each real number stored by Alice and Bob is kept with accuracy of $\log n$ bits.
	The protocol will be the following: Alice chooses a random index $1\leq\hat{i}\leq k$, and then finds the index $1\leq \hat{j}\leq m$ such that
	$$\max_{1\leq j\leq m}\dprod{\theta_j^{\hat{i}}}{u}=\dprod{\theta_{\hat{j}}^{\hat{i}}}{u}.$$
	Alice sends Bob both indices $\hat{i}$ and $\hat{j}$ using at most $\log{k}+\log{m}=(1+C_2)\sqrt{n}$ bits.
	Bob checks the distance of $\theta_{\hat{j}}^{\hat{i}}$ to $H$ and $H^{\perp}$. If $d(\theta_{\hat{j}}^{\hat{i}},H)>d(\theta_{\hat{j}}^{\hat{i}},H)$ then they answer that $u\in H$ otherwise they answer that $u\in H^{\perp}$.\\
	
	In this protocol Alice preforms one measurement. This measurement is in a subspace of dimension $O(\sqrt{n})$, hence the protocol has total rank of $O(\sqrt{n})$.\\
	
	The analysis of this protocol is done in two steps. First we show that the protocol works when we replace $(E_1,\mathcal{N}_1),\dots,(E_k,\mathcal{N}_k)$ with shared random pair $(E,\mathcal{N})$. 
	The complexity of the public coin protocol is $\log m=C_2\sqrt{n}$ bits.
	Second, we eliminate the need for shared randomness by considering $(E_1,\mathcal{N}_1),\dots,(E_k,\mathcal{N}_k)$.
	This step is standard, and the cost of eliminating the shared randomness is another $\log k=\sqrt{n}$ bits.
	We present the main ideas of these steps.\\
	
	In the analysis of the first step, we use two standard results:
	In the first, we use the fact that the norm of a projection of a random vector is close to Gaussian \cite{johnson+lindenstrauss84}. 
	\begin{proposition}\label{prop:random_projection}
		Let $v\in S^{d-1}$ be a random vector distributed uniformly. Let $F\subseteq\bR^d$ be a subspace of dimension $\ell$. Then,
		$$\bP\left(\left||\proj_Fv|^2-\frac{\ell}{d}\right|\geq t\right)\leq Ce^{-ct^2d},\quad\forall t.$$
	\end{proposition}
	Note that by applying a random rotation we may assume that $v$ is fixed and $F$ is random.\\
	The second standard result shows that our choice of $\mathcal{N}_i$ is typically an $1/2-$net of $S^{n-1}\cap E_i$.
	\begin{proposition}\label{prop:net}
		Let $z_1,...,z_{\ell}$ be independent uniformly chosen random vectors in $S^{d-1}$, where $\ell=e^{Cd}$. Then with probability greater than $1-e^{-e^{c\ell}}$ they form an $1/2-$net of the sphere.
	\end{proposition}
	\begin{proof}[Sketch of the proof]
		Let $N$ be an $\varepsilon-$net with $\#N=e^{c_1k}$ (e.g \cite{pisier89}). For any $x\in N$ we have
		$$\bP\left(|z_i-x|>\varepsilon\;\forall i\right)
		\leq e^{-m\bP\left(|z_1-x|\leq\varepsilon\right)}.$$
		Since
		$$\bP\left(|z_1-x|\leq\varepsilon\right)\approx e^{-c_2(1-\varepsilon)^2k},$$
		we have
		$$\bP\left(\exists x\in N;\;|z_i-x|>\varepsilon\;\forall i\right)\leq 
		\exp\left(c_1k-e^{(c-c_2(1-\varepsilon)^2)k}\right).$$	
	\end{proof}
	
	We are now ready to prove that the protocol works with a shared random pair $(E,\mathcal{N})$.
	\begin{proof}
		Let $u,H$ be fixed, such that either $u\in H$ or $u\in H^{\perp}$.
		We have,
		\begin{align}\label{eq:max_projection}
			\max_{\theta\in S^{n-1}\cap E}\dprod{u}{\theta}&=\max_{\theta\in S^{n-1}\cap E}\dprod{u}{\textrm{Proj}_E\theta}
			=\max_{\theta\in S^{n-1}\cap E}\dprod{\textrm{Proj}_E u}{\theta}\nonumber
			=|\textrm{Proj}_E u|.\nonumber
		\end{align}
		
		The dimension of $E$ is $\lfloor C_1\sqrt{n}\rfloor$. By Proposition \ref{prop:random_projection} and (\ref{eq:max_projection}) we can choose $C_1$ big enough such that
		$$\bP\left(\max_{\theta\in S^{n-1}\cap E}\dprod{u}{\theta}\geq \frac{100}{n^{1/4}}\right)\geq 0.91.$$
		Let $\theta_{\hat{j}}\in \mathcal{N}$ be the closest point to $u$. By Proposition \ref{prop:net}, with probability greater than $0.99$, the set $\mathcal{N}=\{\theta_1,\dots,\theta_m\}$ is an $1/2-$net of $S^{n-1}\cap E$. Hence, for any $\theta\in S^{n-1}\cap E$ there exists $\theta_i\in\mathcal{N}$ such that $|\theta_i-\theta|\leq 1/2$. Hence,
		\begin{align*}
			\dprod{\theta}{u}&=\dprod{\theta-\theta_i}{u}+\dprod{\theta_i}{u}
			\leq |\theta-\theta_i||\textrm{Proj}_E u|+\max_{j}\dprod{\theta_j}{u}
			\leq \frac{1}{2}|\textrm{Proj}_E u|+\max_{j}\dprod{\theta_j}{u}.
		\end{align*}
		The right hand side does not depend on $\theta$, hence,
		\begin{align*}
			\max_{j}\dprod{\theta_j}{u}&\geq \max_{\theta\in S^{n-1}\cap E}\dprod{\theta}{u}-\frac{1}{2}|\textrm{Proj}_E u|
			=\frac{1}{2}\max_{\theta\in S^{n-1}\cap E}\dprod{\theta}{u}.
		\end{align*}
		Let $\alpha=\dprod{\theta_{\hat{j}}}{u}$. With probability greater than $0.9$ we have
		$$\alpha\geq \frac{50}{n^{1/4}}.$$
		Let
		$$\theta_{\hat{j}}=\alpha u+\sqrt{1-\alpha^2}v,$$
		where $v\in S^{n-1}\cap u^{\perp}$. By the definition $v$, it is distributed uniformly in $S^{n-1}\cap u^{\perp}$. By Proposition \ref{prop:random_projection} we have
		$$\bP\left(\left||\proj_Hv|^2-\frac{1}{2}\right|\geq \frac{10}{\sqrt{n}}\right)\leq 0.1.$$
		Hence, if $u\in H$, then with probability greater than $0.8$ we have,
		$$|\textrm{Proj}_H \theta_{\hat{j} }|^2=\alpha^2+(1-\alpha^2)|\textrm{Proj}_H v|^2\geq \frac{1}{2}+\frac{1000}{\sqrt{n}},$$
		and
		$$|\proj_{H^{\perp}} \theta_{\hat{j}}|^2\leq |\textrm{Proj}_{H^{\perp}} v|^2 \leq \frac{1}{2}+\frac{10}{\sqrt{n}}.$$	
		Hence, with probability greater than $0.8$ we have $|\proj_{H^{\perp}} \theta_{\hat{j}}|<|\textrm{Proj}_H \theta_{\hat{j} }|$, thus the protocol would correctly determine that $u\in H$.
		The case $u\in H^{\perp}$ is proven similarly.
	\end{proof}
	
	Next we explain how to eliminate the shared randomness.
	\begin{proof}
		We denote by $\theta_j\in\mathcal{N}$ the closest vector to $u$ in $\mathcal{N}$. Let
		$$A=\left\{(u,H,E,\mathcal{N});\;
		\begin{matrix}
		|\textrm{Proj}_H\theta_j|^2> |\textrm{Proj}_{H^{\perp}}\theta_j|^2+10/\sqrt{n},\;
		\textrm{if } u\in H\\
		|\textrm{Proj}_{H^{\perp}}\theta_j|^2> |\textrm{Proj}_H\theta_j|^2+10/\sqrt{n},\;
		\textrm{if } u\in H^{\perp}
		\end{matrix}\right\}.$$
		Let $A_{u,H}$ and $A_{E,\mathcal{N}}$ denote the corresponding sections of the set $A$.
		By the previous step of shared randomness, for any fixed $(u,H)$ we have,
		$$\bP_{(E,\mathcal{N})}\left((u,H)\in A_{E,\mathcal{N}}\right)\geq 0.8.$$
		By the Chernoff-Hoeffding inequality, for any fixed $u,H$ we have
		$$\bP\left(\frac{\#\{i;\;(E_i,\mathcal{N}_i)\in A_{u,H}\}}{m}\leq 0.8\right)\leq e^{-ck}.$$
		Hence, by Fubini's theorem, for most choices of $(E_1,\mathcal{N}_1),\dots,(E_k,\mathcal{N}_k)$
		$$\bP_{u,H}\left(\frac{\#\{i;\;(E_i,\mathcal{N}_i)\in A_{u,H}\}}{m}\leq 0.8\right)\leq e^{-c'k}.$$
		Recall that Alice and Bob sample in advance the list $(E_1,\mathcal{N}_1),\dots,(E_k,\mathcal{N}_k)$. 
		Thus, with high probability, their protocol works for a any $(u,H)$ outside a set of measure $e^{-c'e^{\sqrt{n}}}$. 
		Hence, for any vector $u$ and a subspace $H$ we can find $u'$ and $H'$ for which the protocol works, $|u-u'|\leq 1/\sqrt{n}$ and $\norm{\proj_H-\proj_{H'}}_{op}\leq 1/\sqrt{n}$.
		Therefore $|\textrm{Proj}_H u- \textrm{Proj}_{H'} u'|\leq 2/\sqrt{n}$ and the protocol works for arbitrary $u$ and $H$.
	\end{proof}
	
	\section{Asymptotic estimates}\label{app:laplace}
	Here we present the proofs of Propositions \ref{prop:small_ball} and \ref{prop:large_dev}.
	
	\begin{proof}[proof of Proposition \ref{prop:small_ball}]
		Assume the event $\norm{\Lambda-I/\sqrt{2}}_{op}\leq C(\sqrt{\alpha_1}+\sqrt{2})/n^{1/4}$ holds true. By Proposition \ref{prop:singular2} this event has probability greater than $1-4e^{-\sqrt{n}}$.
		Define $\psi:\bR^{n}\to\bR$ by
		\begin{align*}
			\psi(x)&=\left(1-\left|\Lambda^{-1}U^T x\right|^2\right)_+^{(n/2-k-2)/4}
			\left(1-\left|\left(I-\Lambda^2\right)^{-1/2}U^Tx\right|^2\right)_+^{(n/2-k-2)/4}.
		\end{align*}
		Using the Taylor expansion 
		$$\log(1-|x|^2)=-|x|^2+O(|x|^4)$$
		for any $|x|<3/4$, we have
		
		\begin{align*}
			\psi(x)=\exp\left(-(n/8-k/4-1/2)\left(\left|U\Lambda^{-1}U^Tx\right|^2+\left|U(I-\Lambda^{2})^{-1/2}U^Tx\right|^2+O(|x|^4)\right)\right),
		\end{align*}
		for any $|x|\leq 1/10$.
		By Corollary \ref{cor:sum} for any $x\in \rho n^{-1/4}B_k$ we have
		\begin{align*}
			\psi(x)=\exp&\left(-(n-k-2)|x|^2/2 +O(\rho^2(\sqrt{\alpha_1}+\sqrt{2})^2) 
			+O(\rho^4)+O\left((\alpha_1+1/\sqrt{n})\rho^2\right)\right).
		\end{align*}
		In Corollary \ref{cor:coefficient} we assumed an upper bound on $\alpha_1$. 
		Under this upper bound assumption we can choose $\rho>0$ small enough, independent of $n$ and any specific choice of $\alpha_1$ such that
		\begin{align*}
			\psi(x)&\geq 0.95\exp\left(-(n/2-k/2-1)|x|^2\right)
			\geq 0.95(1-|x|^2)^{(n-k-2)/2},\quad \forall x\in \rho n^{-1/4}B_k.
		\end{align*}
	\end{proof}
	
	\begin{proof}[proof of Proposition \ref{prop:large_dev}]
		By the assumption on $f$ we have
		$$I\leq C_{n,k}e^{\alpha_2\sqrt{n}}\int_{\bR^k\setminus \rho n^{-1/4}B_k}\left(1-|x|^2\right)_+^{(n-k-2)/2}dx$$
		Using $1-x\leq e^{-x}$ and the assumption $k\leq\alpha_1\sqrt{n}$ and that $\alpha_1$ is bounded by some universal constant, for $n$ exceeding some universal constant, we have
		\begin{align*}
			I&\leq C_{n,k}e^{\alpha_2\sqrt{n}}\int_{\bR^k\setminus \rho n^{-1/4}B_k}e^{-(n/2-k/2-1)|x|^2}dx
			\leq C_{n,k}e^{\alpha_2\sqrt{n}}\int_{\bR^k\setminus \rho n^{-1/4}B_k}e^{-n|x|^2/3}dx.
		\end{align*}
		By integrating in polar coordinates, we have
		\begin{align*}
			I&\leq C_{n,k}k\vol{B_k}e^{\alpha_2\sqrt{n}}\int_{\rho n^{-1/4}}^{\infty}r^{k-1}e^{-nr^2/3}dr
			= C_{n,k}k\vol{B_k}e^{\alpha_2\sqrt{n}}\frac{1}{n^{k/2}}\int_{\rho n^{1/4}}^{\infty}r^{k-1}e^{-r^2/3}dr.
		\end{align*}
		Define $h(r)=-(k-1)\log r+r^2/3$. The function $h$ is convex, hence
		$$h(r) \geq h\left(\rho n^{1/4}\right)+h'\left(\rho n^{1/4}\right)\left(r-\rho n^{1/4}\right).$$
		Assuming $\alpha_1\leq \rho^2/6$ we have,
		$$h'\left(\rho n^{1/4}\right)=-\frac{k-1}{\rho n^{1/4}}+\frac{2}{3}\rho n^{1/4}\geq \frac{1}{2}\rho n^{1/4}.$$
		We have,
		\begin{align*}
			\int_{\rho n^{1/4}}^{\infty}e^{-h(r)}dr&\leq e^{-h\left(\rho n^{1/4}\right)}\int_{\rho n^{1/4}}^{\infty}e^{-\rho n^{1/4}\left(r-\rho n^{1/4}\right)/2}dr
			=\frac{2}{\rho n^{1/4}}e^{-h\left(\rho n^{1/4}\right)}.
		\end{align*}
		Hence,
		$$I\leq 2C_{n,k}\vol{B_k}\rho^{k-2}kn^{-k/4-1/2}e^{-\sqrt{n}(\rho^2/3-\alpha_2)}.$$
		Using
		\begin{align*}
			C_{n,k}\textrm{Vol}(B_k)&=\frac{n-k}{n}\frac{\textrm{Vol}(B_{n-k})\textrm{Vol}(B_k)}{\textrm{Vol}(B_n)}
			=\frac{n-k}{n}\binom{n/2}{k/2}\leq \left(\frac{n\cdot e}{k}\right)^{k/2},
		\end{align*}
		we have
		$$I\leq \frac{2}{\rho^2}(\sqrt{e}\rho)^k\frac{n^{k/4-1/2}}{k^{k/2-1}}e^{-\sqrt{n}(\rho^2/3-\alpha_2)}.$$
		Assuming $\alpha_1>0$ is small enough such that $\rho^2/3-\alpha_1\log(\rho\sqrt{e/\alpha_1})>0$, we optimize over $k$, and get
		$$I\leq\frac{2\alpha_1}{\rho^2}\textrm{exp}\left[-\sqrt{n}\left(\rho^2/3-\alpha_1\log(\rho\sqrt{e/\alpha_1})-\alpha_2\right)\right].$$	
		Hence, we can choose
		$$2\alpha_2<\rho^2/3-\alpha_1\log(\rho\sqrt{e/\alpha_1}),$$
		to finish the proof.
	\end{proof}

\bibliography{communication}
\bibliographystyle{plain}

\end{document}